\documentclass[amscd,amssymb,verbatim,12pt]{amsart}
\usepackage[all]{xy}
\usepackage{amssymb,latexsym, amsmath, amscd, array, graphicx}
\usepackage{pdfsync}

\swapnumbers \numberwithin{equation}{section}

\theoremstyle{plain}

\newtheorem{thm}{Theorem}[section]
\newtheorem{theorem}[thm]{Theorem}
\newtheorem{lem}[thm]{Lemma}

\newtheorem{prop}[thm]{Proposition}
\newtheorem{cor}[thm]{Corollary}
\newtheorem{Def}[thm]{Definition}

\theoremstyle{definition}
\newtheorem{defin}[thm]{Definition}

\newtheorem{question}[thm]{Question}

\DeclareMathOperator{\dist}{{\rm dist}}

\def\C{{\mathbb C}}
\def\Z{{\mathbb Z}}
\def\Q{{\mathbb Q}}
\def\R{{\mathbb R}}
\def\N{{\mathbb N}}
\def\H{{\mathbb H}}

\def\inn{\,\in\,}

\def\1{\hbox{\rm\rlap {1}\hskip.03in{\rom I}}}
\def\Bbbone{{\rm1\mathchoice{\kern-0.25em}{\kern-0.25em}
{\kern-0.2em}{\kern-0.2em}I}}

\long\def\forget#1\forgotten{} %

\newcommand\ver[1]{\marginpar{\tiny Changed in Ver \VER}}

\date{\today}
\begin{document}

\title[On cohomology of the Higson compactification]{On cohomology of the Higson compactification of hyperbolic spaces }

\author[A.~Dranishnikov]{Alexander  Dranishnikov$^{1}$}

\author[Th.~Gentimis]{Thanos  Gentimis}
\thanks{$^{1}$Supported by NSF, grant DMS-0904278}

\address{Alexander N. Dranishnikov, Department of Mathematics, University
of Florida, 358 Little Hall, Gainesville, FL 32611-8105, USA}
\email{dranish@math.ufl.edu}



\begin{abstract}
We show that  in dimensions $>1$ the cohomology groups  of the Higson compactification of the hyperbolic space
$\H^n$
with respect to the $C_0$ coarse structure are trivial.

Also we prove that the cohomology groups of the Higson compactification of $\H^n$ for the bounded coarse structure are trivial in all  even dimensions.
\end{abstract}

\maketitle \tableofcontents

\section {Introduction}
The coarse Baum-Connes conjecture  is a coarse analog
of the famous Baum-Connes conjecture ~\cite{HR}. It states that the coarse $K$-theory assembly map $\mu:K_*^{lf}(X)\to K_*(C^*(X))$
is an isomorphism for uniformly contractible metric spaces $X$ with bounded geometry  where the recipient of the coarse assembly map $\mu$
is
the K-theory of the Roe $C^*$-algebra of $X$. The Roe algebra admits an abelian approximation which fits into the commutative diagram
  \[
    \xymatrix{ K_*^{lf}(X) \ar[rr]^{\mu}\ar[dr]_{\mu'}& &K_*(C^*(X))�\\
    & K_{\ast-1}(\nu X)\ar[ur] &
    }.
    \]

The approximation $\mu'$ of the assembly map coincides with the boundary homomorphism in the exact sequence of the pair
$(\overline X,\nu X)$ where $\overline X=X\cup\nu X$ is the Higson compactification of $X$ and $\nu X$ is the 
Higson corona. We recall that the Higson compactification is the maximal ideal space for the $C^*$ algebra generated
by bounded functions with the gradient tending to zero at infinity.

The $K$-theory acyclicity of $\overline X$ would imply that $\mu'$ is isomorphism and therefore, $\mu$ is a monomorphism.
The injectivity of the coarse assembly map is a coarse analog of the analytic  Novikov conjecture.
The coarse Novikov conjecture is of a great interest  for many applications to the classical Novikov type conjectures.
Moreover, the rational acyclicity would be sufficient for many applications. Since the compact space $\overline X$ is not metrizable, it is preferable to
consider the acyclicity with respect to the  cohomology. Note that  cohomology groups are
well defined  for any spectrum and for all compact spaces and they satisfy the Steenrod-Eilenberg axioms.

The statement about the acyclicity of $\overline X$ for universal covers of aspherical manifolds appeared in~\cite{Ro1}
under the name of {\em the Higson Conjecture}. It turns out that the conjecture is false in dimension 1~\cite{K}.
Still, for some of the applications it would suffices to have the acyclicity  in the top dimension.
It turns out that even this version of the Higson Conjecture
is false for the euclidean spaces: $H^n(\overline{\R^n};\Q)\ne 0$ \cite{DF}. On the other hand it was proven to be true \cite{DF}  for some $n$  for the $n$-dimensional hyperbolic spaces $\H^n$.
In \cite{DF} a trick using Hopf bundles $S^{2n-1}\to S^n$ was introduced to
show that $H^n(\overline{\H^n})=0$ for
$n$= 2, 4, and 8. Namely, every $n$-dimensional cohomology class of $\overline\H^n$ can be represented
by a map $f:\overline\H^n\to S^n$. The geometry of $\H^n$ allows to construct a 
lift $g:\H^n\to S^{2n-1}$ of the restriction $f|_{\H^n}$ with respect to the Hopf bundle which is 
slowly oscillating and hence is extendible to the 
Higson corona, $\bar g:\overline\H^n\to S^{2n-1}$. Since $\dim\overline\H^n=n< 2n-1$ (see \cite{DKU}), it follows that $\bar g$ 
and, hence, $f$ are nullhomotopic.

The main point of the above bundle trick, besides the geometry of $\H^n$,
is that the total space of the Hopf bundle is compact and
has higher connectivity than the base. One can play this trick rationally.
Namely, using the spherical tangent bundle to even dimensional spheres
one can obtain that  $ H^n(\overline{\H^n};\Q)=0$ for all even $n$.  This together with
the fact that  $\bar H^*(\overline{\H^n};F)=0$ for finite coefficients (\cite{DFW}),
implies  that for even $n$, $\bar H^n(\overline{\H^n})=0$. 
The following proposition sets the limits of the bundle trick.

\begin{prop}\label{odd}
For odd $n$ there is no fibration $p:X\to K$  between finite CW complexes such that $K$ is
$(n-1)$-connected with $rank(\pi_n(K))\ne 0$ and
$\pi_i(X)\otimes\Q=0$ for $i\le n$.
\end{prop}
\begin{proof}
We show it for $r=rank(\pi_n(K))=1$. The general case can be done by
a minor modification. Let $p:X\to K$ be such a fibration with fiber $F$.
We fix $f:S^n\to K$ which defines an element
of the $\pi_n(K)$ of infinite order. Consider the pull-back
fibration $p':X'\to S^n$. The comparison of the homotopy exact
sequence of fibrations $p$ and $p'$ shows that $\pi_i(X')\otimes\Q=0$ for
$i\le n$. A lift of the natural homotopy $\Omega S^n\times[0,1]\to S^n$ defines a 
map $q:\Omega S^n\to F$ with the homotopy fiber $\Omega X'$.
Then the corresponding fibration $\Omega S^n\stackrel{\Omega X'}\to F$
induces a rational isomorphism of the  $(n-1)$-dimensional homotopy
groups. By the modulo torsions Whitehead theorem~\cite{Sp} we
obtain that $q_*:H_{n-1}(\Omega S^n;\Q)\to H_{n-1}(F;\Q)$ is an
isomorphism. Therefore, $q^*:H^{n-1}(F;\Q)\to H^{n-1}(\Omega S^n;\Q)$ is
an isomorphism. Note that for odd $n$, $H^*(\Omega S^n;\Q)=\Q[x]$
where $deg(x)=n-1$ \cite{H}. Let $y$ be the generator
of $H^{n-1}(F;\Q)$ which corresponds to $x$ by the above
isomorphism. Since $F$ is compact, $y^m=0$ for some $m$. Then
$x^m=0$. Contradiction.
\end{proof}
In this paper we prove (Theorem~\ref{main2}) that $\bar H^k(\H^n)=0$ for all $n$ for all even $k$.
We don't know how to treat the odd dimensional case. Besides, it is known that
$H^1(\overline{\H^n};\Q)\ne 0$ for all $n$.

In ~\cite{Ro2} the Higson compactification was defined for any coarse structure so that
the classic Higson compactification is the Higson compactification with respect to the 
{\em bounded coarse structure} on a metric space. Another natural coarse structure on  metric spaces
which already found applications in geometry and topology ~\cite{W1},\cite{W2} is the $C_0$ coarse structure.
It turns out that  the Higson compactification $h_0\H^n$ of the hyperbolic space $\H^n$
with respect to the $C_0$ coarse structure
is acyclic in all dimensions $>1$ (Theorem~\ref{main1}). To prove this result we use the techniques of the
$\ell_{\infty}$-cohomology.

\section{$\ell_{\infty}$- cohomology}

The following definition is taken from~\cite{Ge}.
A norm on an abelian group $A$ is a non-positive function $|\ |: A\to \R$
such that

1. $|a|=0$ iff $ a=0$;

2. $|a|=|-a|$;

3. $|a+b|\le |a|+|b|$.

Let $A$ be a normed abelian group and $X$ a CW complex. Let $E_n(X)$
denote the set of $n$-cells in $X$ and
$C^n(X,A)=Hom(\oplus_{E_n(X)}\Z,A)$ denote the group of cellular
$n$-cochains on $X$ with value in $A$. Let
$$C_{(\infty)}^n(X,A)=\{\phi\in C^n(X,A)\mid\ \exists\ b\ :\
|\phi(e)|\le b\ \forall\ e\in E_n(X)\}$$ be the subgroup generated by
bounded cochains. If one takes the $\ell_1$ norm on
$C_n(X)=\oplus_{E_n(X)}\Z$ then the group of bounded cochains
consists of homomorphisms $\phi:C_n(X)\to A$ bounded with respect to
the norms. We denote the corresponding cohomology groups by
$H_{(\infty)}^*(X;A)$. The bounded value cohomology groups for the
augmented chain complex $$\dots \to C_n(X)\to \dots\to
C_0(X)\to\Z\to 0$$ are called the reduced $\ell_{\infty}$ cohomology groups
and denoted by $\bar H_{(\infty)}^*(X;A)$.

If a group $A$ is finitely generated then, clearly, $H_{(\infty)}^*(X;A)$ does not depend on the choice
of the norm on $A$.

REMARK. We note that the $\ell_{\infty}$-cohomology differs from the bounded cohomology defined by
Gromov ~\cite{G1}. The difference is that the latter is defined by means of bounded singular cochains whereas
the former is defined in terms of bounded cellular cochains.

The following Proposition  was proven in \cite{Ge}.

\begin{prop}\label{ger}
Let $X$ be the universal cover of $K(\pi,1)$ with finite skeletons
$K(\pi,1)^{(n)}$ for all $n$. Then the inclusion $\Z\to\R$ induces an isomorphism
$$\bar H^i_{(\infty)}(X;\Z)\to \bar H^i_{(\infty)}(X;\R)$$
for $i\ge 0$.
\end{prop}
\begin{proof}
The result follows from the facts that the group
$S^1=\R/\Z$ is bounded with respect
to the quotient norm, the equality
$$\bar H^i_{(\infty)}(X;S^1)
=\bar H^i(X;S^1)=0$$ for $i\ge 0$, and the Coefficient Long Exact
sequence.
\end{proof}

For a group $\Gamma$ with finite complex $K(\Gamma,1)^{(n)}$ we define
$$H_{(\infty)}^i(\Gamma;A)=H_{(\infty)}^i(X;A)$$ where $X$ is the universal cover.
It was shown in \cite{Ge} that the cohomology group does not depend on choice of $K(\Gamma,1)$
with that finiteness condition. We note that for a hyperbolic group $\Gamma$ there is a complex
$K(\Gamma,1)$ with  $K(\Gamma,1)^{(n)}$ finite for all $n$ \cite{BH}.

\begin{thm}\label{hypgr}
For every hyperbolic group $\Gamma$, $H^i_{(\infty)}(\Gamma)=0$ for all $i>1$.
\end{thm}
This theorem was proven in \cite{Mi}  for coefficients in $\R$. In view of   Proposition~\ref{ger}
it holds true for  coefficients in $\Z$.

\section{The $C_0$ Higson Compactification}

\begin{Def} Let $X$ be a set and consider the product $X\times X$. A collection of sets $\mathcal E =\{E\}\subseteq P(X\times X)$ is a {\em coarse structure} (\cite{HR},\cite{Ro2}) if it satisfies the following conditions:
\begin{itemize}
\item It contains the diagonal $D=\{(x,x): x\inn X\}$.
\item If $A\inn \mathcal E$, $B\subseteq A$ then $B\inn \mathcal E$.
\item If $A\inn \mathcal E$ then the set $A^{-1}=\{(y,x):(x,y)\inn A\}\inn\mathcal E$.
\item If $A,B\inn \mathcal E$ then  $A\circ B = \{(x,y): \exists z\inn X \,\text{with}\, (x,z)\inn A\,\text{and}\  
(z,y) \inn B\}\inn\mathcal E.$ 
\item If $A,B\inn \mathcal E$ then $A\cup B\inn \mathcal E$.
\end{itemize}
\end{Def}

EXAMPLE. The {\em bounded coarse structure} $\mathcal E_b$ on a metric space $X$ is the collection of  all sets 
$E\subset X\times X$ lying in a bounded distance from the diagonal.

\begin{defin}\cite{Ro2} Let $(X,d)$ be a metric space. The {\em $C_0$ coarse structure} on $X$ consists of all sets $E\subset X\times X$ such that
$\forall \epsilon>0\,
\ \exists\ K\subseteq X$, a compact subset with $d(x,y)<\epsilon$ for all $(x,y)\inn E\setminus(K\times K)$.
\end{defin}

\begin{defin} Let $(X,d)$ be a metric space with the $C_0$ coarse structure. Let $f:X\to \C$ be a bounded, continuous function. Denote by $\textbf{d} f$ the function: \[ \textbf{d} f:X\times X \to \C\] defined by the formula $f(x,y)=f(x)-f(y)$.

We will say that $f$ is a Higson function if for every controlled set $E$ the restriction of $\textbf{d} f$ to E vanishes at infinity.
\end{defin}

Let $C_{h_0}(X)$ denote the set of all bounded, continuous Higson functions.

\begin{defin} The compactification $h_0X$  of a metric space $X$ characterized by the property $C(h_0X)=C_{h_0}(X)$ is called the $C_0$ Higson compactification. Its boundary $h_0X\setminus X$ is denoted by $\nu_0 X$ and is called the {\em $C_0$ Higson corona}.
\end{defin}
Note that in the definition of the $C_0$ Higson compactification  the complex numbers $\C$ can be replaced by the reals $\R$.

We use notation $\|x\|=d_X(x,x_0)$ where $X$ is a metric space with a based point.

We recall that a map $f:X\to Y$ between metric spaces is called uniformly continuous if there is a monotone function $\omega:[0,\infty)\to[0,\infty)$
called {\em modulus of continuity} with $\lim_{t\to 0}\omega(t)=0$ such that $d_Y(f(x),f(x'))\le\omega(d_X(x,x'))$ for all $x,x'\in X$.

\begin{thm}\label{unifTh} Let $X$ be a proper geodesic metric space. Then $f\inn C_{h_0}(X)$ if and only if it is uniformly continuous and bounded.
\end{thm}
\begin{proof} Let $f\in  C_{h_0}(X)$ and assume that $f$ is not uniformly continuous. Hence there is $\epsilon>0$ and sequences $x_n$ and $y_n$
with $d(x_n,y_n)<1/n$ and with $|f(x_n)-f(y_n)|\ge \epsilon$. Clearly, $E=\{(x_n,y_n)\}$ is a controlled set. We obtain a contradiction with the condition
$\textbf{d}f|_E\to 0$.

Let $f:X\to\C$ be a uniformly continuous bounded function with a modulus of continuity $\omega$
and let $E$ be a controlled set for $C_0$. Then the inequality
 $|f(x)-f(x')|\le\omega(d_X(x,x'))$ implies that $\textbf{d}f|_E\to 0$.
\end{proof}
\begin{cor}\label{uniform}
A map to a compact metric space $f:X\to Y$ continuously extends to the corona $\nu_0X$ if and only if it is uniformly continuous.
\end{cor}
\begin{proof} Let $j:Y\to I^{\omega}$ be an imbedding. Clearly, $j$ is uniformly continuous.
Suppose that $f:X\to Y$ is uniformly continuous. Then every function $f_i=\pi_i\circ j\circ f$ is uniformly continuous
as a composition of such where $\pi_i:I^{\omega}\to i$ is the projection to the $i$th coordinate. Then every $f_i$ is extendible to
$\nu_0X$. These extensions together define a continuous extension $\bar f:h_0X\to Y$ of $j\circ f$.

Assume that $f$  is extendible over $\nu_0X$ and it is not uniformly continuous. Hence there is $\epsilon>0$ and sequences $x_n$ and $y_n$
with $d(x_n,y_n)<1/n$ and with $|f(x_n)-f(y_n)|\ge \epsilon$. In view of compactness of $Y$ and using a subsequences we may assume that
$\lim f(x_n)=u$ and $\lim f(y_n)=v$. Let $\phi:Y\to[0,1]$ be a continuous function with $\phi(u)=0$ and $\phi(v)=1$. Note that
$\phi\circ f$  is  extendable to the corona $\nu_0X$. By Theorem~\ref{unifTh} it is uniformly continuous. We obtained a contradiction.
\end{proof}

We make use of the following fact~\cite{Ra}.
\begin{thm}
Every uniformly continuous bounded function on a metric space  $f:X\to\R$ is a uniform limit of Lipschitz functions.
\end{thm}
\begin{cor}\label{uniform limit}
Every uniformly continuous  map of a metric space  $f:X\to K$  to a finite complex is a uniform limit of Lipschitz maps.
\end{cor}
\begin{proof}
Let $K\subset\Delta^{N-1}\subset\R^N$ be realized in a unit simplex and let $r:W\to K$ be a retraction of a regular neighborhood. We may assume that
$r$ is 2-Lipschitz. By the above theorem $f=(f_1,\dots, f_N)$ where $f_i=\lim g^i_k$ with Lipschitz functions $g^i_k$.
For sufficiently large $k$ we define $g_k=r\circ(g^1_k,\dots,g^N_k)$ . Clearly $f$ is the uniform limit of $g_k$
and each $g_k$ is Lipschitz.
\end{proof}

\section{Lipschitz obstruction theory}

We recall the main theorem of obstruction theory in the simply connected case~\cite{DK}.

\begin{thm}\label{class} Let $L(X,A)$ be a relative CW complex and let $K$ be a  simply connected complex.
Let $f_n:X^{(n)}\cup A\to K$ be a continuous map. 

1. There is a cellular cocycle $\Theta(f)\in C^{n+1}(X,A;\pi_n(K))$
which vanishes if and only if $f$ extends to a map $f_{n+1}:X_{n+1}\to K$.

2. The cohomology class $[\Theta(f)]\in H^{n+1}(X,A;\pi_n(K))$ vanishes if and only if the restriction
$f_{n-1}$ of $f_n$ to $X^{(n-1)}\cup A$ extends to a map $f_{n+1}:X_{n+1}\to K$.
\end{thm}

A similar theorem holds true for a Lipschitz extension problem. To make a precise statement
we consider a geodesic metric on CW complexes such that there are finitely many isomorphism types
of cells. We call such metric {\em unifom}.
The following proposition shows that the obstruction cocycle in the case of Lipschitz maps is  bounded.
\begin{prop}\label{lip-norm} Let $f:S^n\to K$ be a $\lambda$-Lipschitz map from the n-dimensional sphere to a finite simplicial complex K. 
Suppose that $\pi_n(K)$ is a normed abelian group. Then $\ \exists\  b>0$ such that $\|f_\ast\|\leq b$.
\end{prop}
\begin{proof}
Consider the space $\lambda\text{-}Map(S^n,K)$ of $\lambda$-Lipschitz maps $g:S^n\to K$. This space is compact. 
Then consider the maps

\[\lambda\text{-}Map(S^n,K)\stackrel{\phi}\to [S^n,K]=\pi_n(K) \stackrel{\|.\|}\to \Z\] where $\phi(g)=[g]$. Clearly, the 
composition $\Phi$ of these maps  is continuous. Since $\lambda\text{-}Map(S^n,K)$ is compact 
$\Phi(\lambda\text{-}Map(S^n,K))$ is compact and thus bounded. 
\end{proof}

\begin{thm}\label{lip}
Let $(X,A)$ be a relative uniform cellular complex and let $K$ be a  simply connected finite complex with a fixed metric.
Let $f_n:X^{(n)}\cup A\to K$ be a Lipschitz map. 

1. There is a bounded cellular cocycle $\Theta_L(f)\in C^{n+1}_{(\infty)}(X,A;\pi_n(K))$
which vanishes if and only if $f$ extends to a Lipschitz map $f_{n+1}:X_{n+1}\to K$.

2. The cohomology class $[\Theta_L(f)]\in H^{n+1}_{(\infty)}(X,A;\pi_n(K))$ vanishes if and only if the restriction
$f_{n-1}$ of $f_n$ to $X^{(n-1)}\cup A$ extends to a Lipschitz map $f_{n+1}:X_{n+1}\to K$.
\end{thm}
\begin{proof}
1. In the case of Lipschitz map $f$ in view of Proposition~\ref{lip-norm} the obstruction cocycle $\Theta(f):C_{n+1}(X,A)\to \pi_n(K)$ is bounded.
So we take $\Theta_L(f)=\Theta(f)$.

2. In the  proof of Theorem~\ref{class}  an extension $f_{n+1}$ is obtained by
construction of a  map $g_n:X^{(n)}\cup A$ that agrees on $X^{(n-1)}\cup A$ with $f$
such that the difference cochain $d_{f,g}=d$ where  $\delta d=\Theta(f)$. Since for the difference cochain
$\delta d_{f,g}=\Theta(f)-\Theta(g)$, we obtain that $g_n$ extends to the $n+1$-skeleton.
When $d$ is bounded, the map $g_n$ can be constructed to be Lipschitz in view of finite choice of homotopy classes.
Then the extension $g_{n+1}$ of $g_n$ which exists by the classical obstruction theory can be taken to be Lipschitz.
\end{proof}

Like in the classical case Theorem~\ref{lip} implies the corresponding theorem about the primary
obstruction for constructing a Lipschitz homotopy between Lipschitz maps.

\begin{thm}\label{obstruction}
Let $f,g:X\to K$ two Lipschitz maps of a uniform complex $X$ to a simply connected
finite complex $K$. Suppose that there is a Lipschitz homotopy $H:X^{(n)}\times I\to K$ with
$$H|_{X^{(n)}\times\{0\}}=f|_{X^{(n)}\times\{0\}}\ \ \ \text{and}\ \ \ H|_{X^{(n)}\times\{1\}}=g|_{X^{(n)}\times\{1\}}.$$
Then there is a Lipschitz map $\bar H:X^{(n+1)}\times I\to K$ that coincides with $H$ on $X^{(n-1)}\times I$ and
with
$$\bar H|_{X^{(n+1)}\times\{0\}}=f|_{X^{(n+1)}\times\{0\}}\ \ \text{and}\ \ \ \bar H|_{X^{(n+1)}\times\{1\}}=g|_{X^{(n+1)}\times\{1\}}$$ if and only if the certain obstruction class $\Theta_{f,g}\in H_{(\infty)}^{n+1}(X;\pi_{n+1}(K))$ is zero.
\end{thm}

\section{Cohomology of the Higson compactification of $\H^n$ for the $C_0$ coarse structure}

We recall that the \v Cech cohomology of a space $X$ can be defined by means of homotopy classes of maps to the Eilenberg-MacLane complex,
$\check{H}^k(X)=[X,K(\Z,n)]$.
\begin{theorem}\label{main1} 
Let $\H^n$ be the n dimensional hyperbolic space and $h_0\H^n$, its $C_0$ Higson compactification.
Then $\check{H}^k(h_0\H^n)=0$ for all $k,n>1$.
\end{theorem}
\begin{proof}
Let $[\bar f]\inn \check{H}^k(h_0\H^n)$ be defined by a map
$\bar{f}:h_0\H^n\to K(\Z,k)$. It is known that the complex  $K(\Z,k)$ can be chosen in such a way that
all its skeletons are  finite complexes. Since  $h_0\H^n$ is compact, the image of $f$ lies in some skeleton $K(\Z,k)^{(i)}=K$ for $i>n$. 
Consider the restriction
$f:\H^n\to K$. Let $\epsilon>0$ be such that any two
$\epsilon$-maps to $K$ are homotopic.
Since $f\inn C_{h_0}$ by Corollary~\ref{uniform} and Corollary~\ref{uniform limit}
there exists a Lipschitz map $g:\H^n\to K$ $\epsilon$-close to $f$. Since $g$ is
Lipschitz, it extends to $\bar g:h_0\H^n\to K$. 

We consider the following equivariant triangulation on $\H^n$: Take a uniform lattice on $\H^n$, consider a triangulation on the orbit manifold and take the lift. We refer to  textbooks on lattices  (say, ~\cite{Mo}) for the existence of such lattices for all $n$.
Let $X=\H^n$ denote the corresponding simplicial complex taken with the uniform metric. 
We note that $X$ is quasi-isometric to $\H^n$. By Theorem~\ref{hypgr}, $H^i_{(\infty)}(X;A)=0$ for $i>1$ for any finitely generated
normed abelian group $A$.
We construct by induction  a sequence of Lipschitz homotopies
$H_i:X^{(i)}\times I\to K$ such that:
\[ H_i|_{X^{(i)}\times \{0\}}=g,\,\,H_i|_{X^{(i)}\times \{1\}}=c \,\,(\text{a constant function}).\] 

A Lipschitz homotopy on the 1-skeleton  $X^{(1)}$ can be easily constructed since $K$ is simply connected.
Assume that $H_i$ is already constructed, $i\ge 1$. Since   $H^{i+1}_{(\infty)}(X;\pi_{i+1}(K))=0$,
Theorem~\ref{obstruction} implies that there is a  required homotopy $H_{i+1}:X^{(i+1)}\times I\to K$.

Note that $X^{(n)}=X$. Let $\lambda$ be a Lipschitz constant for the map $H=H_n$.
We consider the associated map \[h:X\to \lambda\text{-}Map(I,K)\] defined by 
$H$. 
By the Ascoli Arzela Theorem, the space $\lambda\text{-}Map(I,K)$ is compact. 
Note that the map $h$ is Lipschitz. By Corollary~\ref{uniform} it admits a continuous extension
$\bar{h}: h_0\H^n\to \lambda\text{-}Map(I,K)$.  Clearly, the associate map $\bar{H}:h_0\H^n\times I\to K$
defines a
homotopy between $\bar{g}$ and a constant map.
Thus, $[\bar{g}]=[0]$. Since $f,g$ are $\epsilon$-close, we obtain
$[\bar{f}]=[\bar{g}]=[0]$. 
\end{proof}

\section{Cohomology of the Higson compactification  of $\H^n$ for the bounded coarse structure}

We recall that the standard Higson compactification of a metric space is defined by means of slowly 
oscillating functions~\cite{DKU}. A map to a locally compact metric space to a compact metric space $f:X\to K$ is called
{\em slowly oscillating} if for every $r>0$, $\lim_{x\to\infty}diam(f(B_r(x)))=0$ where $B_r(x)$  denotes the $r$-ball centered at $x$.

Let $Lip(f:X\to Y)=\sup\{\frac{d_Y(f(x),f(x'))}{d_X(x,x')}\}$.
We call a function $f$ to be {\em $r$-locally $\lambda$-Lipschitz} if  $$d_Y(f(x),f(x'))\le\lambda d_X(x,x')$$ for all $x,x'\in X$ with $d_X(x,x')\le r$. 
Thus, it is $\lambda$-Lipschitz on every $r/2$-ball,
$Lip(f|_{B_{\frac{r}{2}} (x)})\le\lambda$. 

\begin{prop}\label{locally}
(1) Suppose that  a bounded function $f:X\to\R$ on a complete Reimannian manifold $X$ satisfies the condition
$\lim_{x\to\infty}Lip(f|_{B_r(x)})=0$ for some $r>0$.Then $f$ is slowly oscillating.

(2) Every  slowly oscillating bounded function $g:X\to\R$ can be uniformly approximated by functions $f$
with $\lim_{x\to\infty}Lip(f|_{B_1(x)})=0$.
\end{prop}
\begin{proof}
(1) is obvious.

(2) It was shown in~\cite{Ro1} that the algebra of functions defining the Higson compactification
is the completion of the space of bounded functions with the gradient tending to zero at infinity.
Clearly every such function $f$ satisfies the condition $\lim_{x\to\infty}Lip(f|_{B_1(x)})=0$.
\end{proof}

We recall that a smooth map $p:P\to M$ between Riemannian manifolds is called {\em Riemannian submersion} if
the induced map of the tangent bundles $Tp:TP\to TM$ preserves the length of horizontal vectors, i.e., vectors orthogonal to
fibers $p^{-1}(y)$. We call a geodesic segment $\gamma:[a,b]\to P$ {\em horizontal at} $t\in[a,b]$ 
if the tangent vector to $\gamma$ at $t$ is horizontal.
\begin{lem}[Hermann~\cite{He}]
Let $p:P\to M$ be a Riemannian submersion. If $\gamma$ is a geodesic in $P$ which is horizontal at one point, then it is always horizontal and
$p\circ\gamma$ is a geodesic in $M$.
\end{lem}
\begin{cor}\label{her}
Every geodesic segment $\gamma:[a,b]\to M$  for every $x\in p^{-1}(a)$ admits a unique horizontal lift $\bar\gamma:[a,b]\to P$ starting at $x$.
\end{cor}
We call a homotopy $H:Z\times[0,1]\to N$ in a Riemannian manifold $N$ {\em geodesic} if for every $z\in Z$, the restriction $H|_{\{z\}\times[0,1]}$
is a uniformly parametrized geodesic segment, i.e., each point $z$ traverses a geodesic segment (perhaps degenerate) with a uniform speed. Corollary~\ref{her} implies the following:\\
{\bf Homotoy Lifting Property.} {\em Let $p:P\to M$ be a Riemannian submersion and let $H:Z\times[0,1]\to M$
be a smooth geodesic homotopy. Then a smooth partial lift $h:Z\times\{0\}\to P$ admits a unique horizontal geodesic 
lift $\bar H:Z\times[0,1]\to P$ that extends $h$.}

For a Reimannian manifold $N$ we denote by $Exp:TN\to N\times N$ the exponential map: $Exp(x,v)=x\times exp_x(v)$, $v\in TN_x$.
\begin{prop}\label{retraction}
Let $p:P\to M$ be a locally trivial fibration between smooth closed manifolds.
Then there is a neighborhood $W$ of the graph $\Gamma_p\subset M\times P$ of $p$ and a smooth map $r:W\to \Gamma_p$ which is a fiberwise retraction, i.e.,
$r(x,y)\in p^{-1}(x)$ for all $(x,y)\in W$.
\end{prop}
\begin{proof}
We note that $P$ and $M$ can be given Riemannian metrics such that $p$ is a Riemannian submersion~\cite{KMS}.
Let  $\epsilon>0$ be smaller than the injectivity radius of $M$. Let $U\subset M\times M$ be the image of the $\epsilon$-ball sub-bundle $TM_{\epsilon}$ of the tangent bundle $TM$ under the exponential map $Exp:TM\to M\times M$.
Thus, $U=\cup_{x\in M}\{x\}\times B_{\epsilon}(x)$ is the union of the $\epsilon$-balls.
Let $R_t:U\to\Delta M$ be the fiberwise geodesic contraction. It defines a fiberwise geodesic homotopy $G:U\times[0,1]\to M\times M$.
Let $W=(1\times p)^{-1}(U)\subset M\times P$. 
Consider a geodesic homotopy $H=(1\times p)\circ G:W\times [0,1]\to M\times M$. By the Homotopy Lifting property, it has a horizonatal lift
to a geodesic homotopy $\bar H:W\times[0,1]\to M\times P$. Then $r=H(\ ,1)$ is a smooth fiberwise retraction of $W$ onto $\Gamma_p$.

Note that the retraction $r$ sends every point $(x,y)\in W$ to $(x,z)$ where $z$ is the nearest to $y$ point in $p^{-1}(x)$.
\end{proof}

The following Proposition is an extension of a construction from the proof of Theorem 5.1 in~\cite{DF}.
\begin{prop}\label{bundle}
Let $h:E\to B$ be a locally trivial  bundle between  closed Riemannian manifolds. Then  for every
map $f:\H^n\to B$ with $\lim_{x\to\infty}Lip(f|_{B_1(x)})=0$ there is a slowly oscillating lift $g:\H^n\to E$.
\end{prop}
\begin{proof}
By Proposition~\ref{retraction} there is a neighborhood $W$  in $B\times E$
of the graph $\Gamma_h$ of $h$ that
admits a fiberwise smooth retraction $p:W_{\epsilon}\to \Gamma$. Thus the compositions with the projections
$\pi_B\circ p$ and $\pi_E\circ p$ are $K$-Lipschitz for some constant $K>0$.
We consider the box metric on $B\times E$, i.e., $$d((x,z),(x',z'))=\max\{d_B(x,x'), d_E(z,z')\}.$$
Then the map $p$ is $K$-Lipschitz. There is $\epsilon>0$  such that the $\epsilon$-neighborhood $N_{\epsilon}(\Gamma_h)$  lies in $W$.

Let $x_0$ be a fixed point in $\H^n$ and let $S(r)$ denote the sphere of
radius $r\in\N$ centered at $x_0$.  Let $\xi_{r}:\H^n\to B(r)$ be the
geodesic retraction onto the $r$-ball $B(r)$ centered at $x_0$.  The hyperbolicity of the metric on $\H^n$
implies that there is a constant $C<1$ such that
$$\xi_r\mid_{S(r+1)}:S(r+1)\to S(r)$$ is a 2-locally $C$-Lipschitz map for all $r$.
Let $m\in\N$ be such that $C^m<1/2K$.

We define a lift $g:\H^n \to E$ of $f$ with respect to $h$ as
follows: Choose a ball $B(r)=B_r(x_0)$ of radius $r$ centered at $x_0$ so that the restriction $f|_{\H^n\setminus B(r)}$ is $m$-locally $1/2K$-Lipschitz and
for every two points $x, y\in \H^n\setminus B(r)$ with
$\dist(x,y)\le m$ the fibers $h^{-1}(f(x))$ and $h^{-1}(f(y))$  are $\epsilon$-close in the Hausdorff metric.
This is possible since $f$ and $h^{-1}\circ f:\H^n\to 2^E$ are both slowly oscillating maps.

We define a lift $g_k:B(r+mk)\to E$ by induction  on $k$.  We begin with any
Lipschitz lift $g_0$ of $f$ over $B(r)$.
Assuming that $g_k$ is already defined on $B(r+mk)$, we define
$$g_{k+1}(x)=pr_2p(f(x),g_k(\xi_{r+mk}(x)))$$
for $x\in B(r+mk+m)\setminus int(B(r+mk))$ where $pr_2:B\times E\to E$ is the projection onto the second factor. Since $d(x,\xi_{r+mk}(x))\le m$, the point
$g_n(\xi_{r+mk}(x))\in h^{-1}f(\xi_{r+mk}(x))$ lies in the $\epsilon$-neighborhood of $h^{-1}(f(x))$. Hence,
the point $(f(x),g_k(\xi_{r+mk}(x)))\in B\times E$ lies in the $\epsilon$-neighborhood  of the fiber $\{f(x)\}\times h^{-1}(f(x))\subset\Gamma$
and, hence, $g_k(x)$ is well-defined. Since $h\circ pr_2\circ p(y,z)=y$ for all $y$ and $z$, we obtain that $g_{k+1}$ is a lift of $f$. Clearly, $g_{k+1}$ extends $g_k$.
The union of all $g_k$ defines a lift $g:\H^n\to E$ of $f$.

Let $\alpha_k$ be a $2$-local Lipschitz constant of $f$ restricted to the complement to to the ball $B(r+mk)$.
Denote the 2-local Lipschitz constant of $g$ restricted to $S(r+mt)$
by $L_t$. 
Since $g_k$ restricted to $S(r+mk)$ is 2-locally $L_k$-Lipschitz and $\xi_{r+mk}$ is 2-locally $C^m$-Lipschitz with $C^m<1$, it follows that the composition $$g_k\circ\xi_{r+mk}:S(r+mk+m)\to E$$ is 2-locally $L_kC^m$-Lipschitz. Therefore, the map $$(f, g_k\circ\xi_{r+mk}):S(r+mk+m)\to B\times E$$ is
2-locally $(\max\{\alpha_k,L_kC^m\})$-Lipschitz. Then the map $p\circ(f, g_k\circ\xi_{r+mk})$ restricted to $S(r+mk+m))$ is 2-locally $K(\max\{\alpha_n,L_nC^m\})$-Lipschitz. Since the projection $pr_2:B\times E\to E$ is 1-Lipschitz,  the map $g_{k+1}$ restricted to
the sphere $S(r+mk+m))$ is 2-locally $K(\max\{\alpha_k,L_kC^m\})$-Lipschitz. Denote by $\beta=KC^m$ and  by $\gamma_k=K\alpha_k$
and note that $\beta,\gamma_k<1$. Let $\gamma_{-1}=L_0$.

Thus,
$$
L_{k+1}\le \max\{\gamma_k, \beta L_k\}\ \eqno{(*)}.
$$
We can prove by induction on $k$ the inequality
$$
L_{k}\le\max\{\beta^i\gamma_{k-i-1}\}_{i=0}^k.
$$
It holds true for $k=0$: $L_0=\gamma_{-1}$. In view of the inequality $(*)$ we have
$$
L_{k+1}\le \max\{\gamma_k, \beta L_k\}\le\max\{\gamma_k,\beta^i\gamma_{k-i-1}\}_{i=0}^k=
\max\{\beta^i\gamma_{k-i}\}_{i=0}^{k+1}.
$$
We show that
for $k\in\N$, $\lim_{k\to\infty}L_k=0$. 
Indeed, if there is a bounded from 0
infinite subsequence $\beta^{i_s}\gamma_{k_s-i_s}>b>0$ with
$k_s\to\infty$, then $i_s$ cannot be bounded from above. Otherwise we would get a contradiction with $\gamma_k\to 0$. If $i_s$ is unbounded, we have
$\beta^{i_s}\to 0$ and hence $\beta^{i_s}\gamma_{k_s-i_s}\to 0$. It is still a contradiction.

Note that for $t\in[k,k+1]$,
$ L_t\le K(\max\{\alpha_k,L_k\})$. Hence $L_{t}\to 0$  as $t
\to\infty$. 

For every geodesic ray $z:\R_+\to\H^n$ issued from $x_0$,  for every $k\in\N$, and $s\in[0,m]$
we have $$\xi_{r+mk}(z(r+mk+s))=\xi_{r+mk}(z(r+mk))=z(r+mk)$$ and hence,
$$d_E(g(z(r+mk+s)),g(z(r+mk)))\le Kd_B(f(z(r+mk+s)),f(z(r+mk))).$$ 
This together with the assumption   $\lim_{x\to\infty}Lip(f|_{B_1(x)})=0$ imply that $Lip(g|_{B_m(x)\cap im(z)})\to 0$ uniformly on $z$.

Since both the spherical and the radial 2-local Lipschitz constants of $g$ tend to zero at infinity, we have
$Lip(g|_{B_1(x)})\to 0$ and Proposition~\ref{locally} implies that $g$ is slowly oscillating.
\end{proof}

\

Everywhere below $\bar H^*$ denotes the reduced Chech cohomology.
\begin{thm}[\cite{DFW}]\label{DFW}
For a uniformly contractible metric space $X$ with bounded geometry
and finite asymptotic dimension
$$
\bar H^*(\overline X;\Z_p)=0
$$
for all $p\in \N$ where $\overline X=X\cup\nu X$ is the standard Higson
compactification
\end{thm}
Since the asymptotic dimension of $\H^n$ is finite~\cite{G2}, ~\cite{Ro2},
we obtain the following
\begin{cor}\label{DFWfor H}
$$
\bar H^*(\overline{\H^n};\Z_p)=0
$$
for all $p$  where $\overline \H^n=\H^n\cup\nu \H^n$ is the standard Higson compactification.
\end{cor}

\begin{thm}\label{main2} 
For all $k$ and $n$
$$\bar H^{2k}(\overline{\H^n})=0.$$
\end{thm}
\begin{proof}
We show that $\bar K^0(\overline{\H^n})=0$. Every element
$\alpha\in \bar K^0(\overline{\H^n})$
in the reduced complex K-theory can be represented by
a map $g:\overline{\H^n}\to BU$. In view of compactness of
$\overline{\H^n}$ there is $m$ such
$g$ lands in  a compact subset $K\subset BU(m)\subset BU$.
There is $\epsilon>0$ such that every two $\epsilon$-close maps to $BU(m)$ are homotopic.
We use Proposition~\ref{locally} to take an $\epsilon$-approximation $f:\overline{\H^n}\to BU(m)$ 
of $g$  with 
the property $\lim_{x\to\infty}Lip(f|_{B_1(x)})=0$. Then $f$ is representing $\alpha$. 
In view of compactness of $U(m)$,
Proposition~\ref{bundle}
applied to $f$ and the pull-back $E'\to K$ of the locally trivial
bundle $E(n)\to BU(n)$ implies
that $f$ factors through a contractible space
and hence is null-homotopic.

The Chern character isomorphism $$\bar
K^0(\overline{\H^n})\otimes\Q\cong \prod_{k}\bar
H^{2k}(\overline{\H^n};\Q)$$ implies that $\bar
H^{2k}(\overline{\H^n};\Q)=0$ for all $n$ and $k$. The Universal
Coefficient Formula  and Corollary~\ref{DFWfor H} imply the result.
\end{proof}

\begin{question}\label{Q-odd}
Is $\check H^i(\overline\H^n)=0$ for odd $i>1$?
\end{question}

\end{document}